\documentclass[a4paper]{article}
\usepackage[font=sf, labelfont={sf,bf}, margin=1cm]{caption}
\usepackage[utf8]{inputenc}
\usepackage{lmodern}
\usepackage{amssymb,amsfonts,amsthm,amsmath,bbm,mathrsfs,aliascnt,mathtools}
\usepackage[english]{babel}
\usepackage[toc,page]{appendix} 
\usepackage{stmaryrd}
\usepackage{verbatim}
\usepackage{hyperref}
\usepackage{color}
\usepackage{bbm}
\usepackage[T1]{fontenc}
\usepackage[cyr]{aeguill}
\date{}

\title{\bf{\textsc{A version of Herbert A. Simon's model \\with slowly fading memory
and  \\ its connections to branching processes}}}

\author{Jean Bertoin\thanks{ Institut f\"ur Mathematik, 
Universit\"at Z\"urich, 
Winterthurerstrasse 190, 
CH-8057 Z\"urich, Switzerland. \hfill \eject
Email: jean.bertoin@math.uzh.ch}  }

\begin{document}

\maketitle

\newtheorem{theo}{Theorem}
\newtheorem{lemma}{Lemma}
\newtheorem{prop}{Proposition}
\newtheorem{cor}{Corollary}
\newtheorem{defi}{Definition}
\newtheorem{rem}{Remark}

\newcommand{\e}{{\mathrm e}}
\newcommand{\rep}{{\mathrm{rep}}}
\newcommand{\R}{{\mathbb{R}}}
\newcommand{\C}{{\mathbb{C}}}
\newcommand{\E}{{\mathbb{E}}}
\newcommand{\TT}{\mathbb{T}}
\newcommand{\dd}{\mathrm{d}}
\newcommand{\N}{\mathbb{N}}
\newcommand{\PP}{\mathbb{P}}
\newcommand{\Q}{\mathbb{Q}}
\newcommand{\Z}{\mathbb{Z}}
\newcommand\Ind{\mathbbm{1}}

\begin{abstract}  Construct recursively a long string of words $w_1\ldots w_n$, such that at each step $k$,
$w_{k+1}$ is a new word  with a fixed probability $p\in(0,1)$, and repeats some preceding word with complementary probability $1-p$. More precisely, given a repetition occurs, $w_{k+1}$ repeats 
the $j$-th word with probability proportional to $j^{\alpha}$ for $j=1,\ldots, k$. We show that the proportion of distinct words occurring exactly $\ell$ times converges as the length $n$ of the string goes to infinity to some probability mass function in the variable $\ell\geq 1$, whose tail  decays as a power function  when $1-p>\alpha/(1+\alpha)$, and exponentially fast when $1-p<\alpha/(1+\alpha)$.
 \end{abstract}

\medskip


\medskip

\section{Introduction}
Partly inspired by the earlier work of G. U. Yule, Herbert A. Simon \cite{Simon} argued in 1955  that an elementary stochastic model could explain the occurrence of power tail distributions in a variety of empirical data. In short, he introduced a simple random algorithm  to produce a long string of words $w_1\ldots w_n$. This algorithm depends on a parameter $p\in(0,1)$ which, in some sense, measures the innovation, and can be described as follows. 
Once the first word $w_1$ has been written, for each $k=1, \ldots,n-1$,  with probability $p$, $w_{k+1}$ is a new word different from all the preceding, and with complementary probability $\bar p=1-p$, $w_{k+1}$ is copied from a uniform sample from $w_1, \ldots, w_k$.

Simon's model can be viewed  the germ of many network growth dynamics with preferential attachment that  flourished since the turn of the millennium. We merely refer here to \cite{BornEb} for an application to the World Wide Web, and to the textbooks \cite{DM2, Durrett, RvdH} and their bibliographies which contain a wealth of relevant references. During the last decade or so, several works in this field have aimed at taking also into account further relevant features for specific models. In particular, an important issue e.g. in citation networks,  is to incorporate  aging, or decay of relevance,  or fading memory phenomena. Dorogovtsev  and Mendes \cite{DM1, DM2}  first generalized the Barab\' asi-Albert  model by letting nodes loosing attraction as their ages increase. 
Cattuto {\em et al.} \cite{CLS} modified Simon's model by assigning different weights to  different words; we shall present their framework explicitly later on as it constitutes the basis of the present work. 
 In a somewhat different direction, the linear rate birth processes related to Yule's model or to classical  preferential attachment dynamics, have been replaced by nonlinear time fractional birth processes to emulate the effect of a slowly-decaying memory in  \cite{ LPS, PPS, Poli}, whereas \cite{GvHW} rather introduces fitness and age-dependency for those birth processes. Further generalizations can be found in the literature,  also refer for instance to  \cite{SSK, Wyuyu}.

The modification Simon's model  that Cattuto {\em et al.} considered concerns the copy mechanism, which is no longer made uniformly at random but rather depends on weights assigned to each word. In general, the weight of a word $w_i$ is a function $q(i,k)$ of  both the rank $i$ of that word and the current total length $k$ of the string. Recall that  the word $w_{k+1}$ is attached at the $k$-th step of the algorithm to the string $w_1 \ldots w_{k}$; just as for the original model of Simon,  $w_{k+1}$ is a new word different from all the preceding with a fixed  probability $p$. However , with probability $\bar p=1-p$, $w_{k+1}$ is a copy of one of the preceding words  sampled now at random with probability proportional to its weight, meaning $w_{k+1}=w_{J}$ for some random index $J$ with distribution
$$\PP(J=j)=\frac{q(j,k)}{\sum_{i=1}^k q(i,k)},\qquad j=1,\ldots, k.$$

We refer henceforth to this modified algorithm as the {\em $q$-weighted Simon's model}.
When weights $i\mapsto q(i,k)$ increase with the index $i\leq k$, recent words are more likely to be repeated than old ones, henceforth emulating a fading memory effect; Simon's original model is recovered plainly when all weights are equal.
Obviously, assigning weights to words also changes profoundly the mathematical analysis of the algorithm. Notably, the process which counts the number of occurrences of some given world as a function of the number of steps of the algorithm is a Markov chain with explicit transition distributions  for Simon's original model, but the Markov property is lost when weights are not all equal.

Cattuto {\em et al.} \cite{CLS}  considered hyperbolic weights,  
viz. $q(i,k)=1/(\tau +k-i)$, where $\tau >0$ is a characteristic time-scale. Quite recently, \cite{SPAB} dealt with short-ranged memory with  $q(i,k)=\Ind_{\{k-i\leq \kappa\}}$, where $\kappa>0$ should be thought of as the memory range.
In the present work, we shall focus on weights depending as a power function on the rank of the word only, 
\begin{equation}\label{E:q}
q(i,k)=q(i)=i^{\alpha}, \qquad \text{for all $1\leq i \leq k$}
\end{equation}
where $\alpha>0$ is another parameter. Note that  for any $a\in(0,1)$, the probability of repeating  some word $w_i$ with  $i\leq ak$ at the $k$-th step tends to $a^{\alpha+1}$ as $k\to \infty$ for the choice \eqref{E:q}, whereas this probability would converge $0$ in Cattuto's setting. Hence, memory fades away more slowly in the present framework than in Cattuto's. 

Our purpose here is not to increase marginally the already rich variety of models that exist in this area, but rather to point out that, despite of the loss of the Markov property, a fairly detailed  analysis can be made thanks to connections with some rather simple branching processes that may be interesting in their own right. A priori, the appearance of branching processes in this setting should certainly not come as a surprise, as it is well-know that they play a fundamental role in preferential attachment dynamics. Nonetheless the connection here is somehow less direct, it only holds asymptotically, and it does not seem straightforward to describe a rigorous construction of weighted Simon's models from branching processes.

To start with, recall two fundamental features of Simon's original model. First,  the proportion of distinct words which are repeated some fixed amount of times  converges as the length of the string goes to infinity. Specifically, 
note that the number of different words which have been used when the string has total length $n$ is close to $np$ when $n\gg 1$. If for every $j\geq 1$,  we write $\nu_n(j)$ for the number of different words which occur exactly $j$ times in a string of total length $n$, then
$$\lim_{n\to \infty} \frac{\nu_n(j)}{np} = \rho {\mathrm B}(j, \rho +1)\qquad \text{a.s. for all $j\geq 1$,}$$
where $\rho=1/\bar p $ and ${\mathrm B}$ stands for the Beta function.
The right-hand side above  is a probability mass function called the Yule-Simon distribution (with parameter $\rho$). 
 Second, since ${\mathrm B}(j, \rho +1) \sim \Gamma(\rho +1) j^{-(\rho+1)}$  as $j\to\infty$;
  the tail  of the Yule-Simon distribution decays as a power function with exponent $-\rho$. 

  We now state a simple version of our main result for the $q$-weighted Simon's model when weights are given by \eqref{E:q}. In short,  as the length of the string tends to infinity,
 the proportion of different words with any given occurrence number converges towards some mass distribution $\varphi$. 
Further,  $p_c=1-\alpha/(1+\alpha)$ is the critical innovation parameter, in the sense that $\varphi$ has a power tail when 
$p<p_c$, and and exponential tail when $p>p_c$. 
  
  \begin{theo}\label{T1}  Let $\nu_n(\ell)$ denote the number of different words which are repeated exactly $\ell$ times in a string of total length $n$ generated by the  $q$-weighted Simon's model, with $q$ given by \eqref{E:q}. Then we have:
   \begin{enumerate}
  \item[(i)] 
   for every $\ell\geq 1$, 
$$\lim_{n\to \infty} \frac{\nu_n(\ell)}{np} = \varphi(\ell)\qquad \text{in probability,}$$
  where $\varphi$ is the probability mass function of some random variable with values in $\N=\{1,2,\ldots\}$,

\item[(ii)]   
   if $\bar p = 1-p >\alpha/(1+\alpha)$, then
  $$\sum_{j>k} \varphi(j) \sim C k^{-1/(\bar p(1+\alpha)-\alpha)}\qquad \text{as }k\to \infty$$
  for some $0<C<\infty$, 
  \item[(iii)] if $\bar p <\alpha/(1+\alpha)$, 
  then there exists $c>0$ with
  $$\sum_{j=1}^{\infty} \e^{cj}  \varphi(j)<\infty.$$
  \end{enumerate}
 \end{theo}
  
  More precisely, we shall describe the probability mass function $\varphi$ that appears in Theorem \ref{T1} in two different ways.
  First, it can be seen as the law of the number of birth events in some continuous state branching process  that occurred before an  independent exponential random time. Equivalently, it can also be seen as the law of a general (Crump-Mode-Jagers) branching process also evaluated at an independent exponential random time. 
  Although we have not been able 
to provide any explicit expressions $\varphi$, we shall establish the power or exponential tail behaviors stated in Theorem \ref{T1} by making use of classical martingales associated to branching processes.

The rest of this work is organized as follows. Section 2 introduces the notion of attraction of a given word as a function of the length of the string. This enables us to circumvent the lack of the Markov property of
 the occurrence counting alone by viewing it as a component of a Markovian pair, and then to analyze its asymptotic behavior. Section 3 describe representations of the limiting processes of Section 2 in terms of  simple branching processes, and describe their behaviors as time goes to infinity using classical martingales in this setting. Finally, Theorem \ref{T1} is proved in Section 4. 

\section{Occurrence counting and attraction}
As it has already been mentioned in the Introduction, an obvious obstacle in the analysis of frequencies of words  for a weighted Simon's model is that the process counting the number of occurrences of a word is not Markov. We shall circumvent this difficulty by considering another natural functional of the algorithm, which enjoys the Markov property and from which occurrence counting can be recovered.

To start with, we say that the $j$-th word in a string is new if $w_j\neq w_i$ for all $1\leq i<j$ (for any $j\geq 2$, this event has probability $p$); otherwise, we say that the $j$-th word
 is a repetition. When the $j$-th word is new, we define its occurrence counting and its attraction  as a function of length of the string, respectively by
 $$N_j(n)= \mathrm {Card} \{1\leq k \leq n: w_k=w_j\}$$
 and 
$$A_j(n)= n^{-\alpha} \sum_{k=1}^n k^{\alpha} \Ind_{\{w_k=w_j\}}.$$
Observe that the second quantity is proportional to the probability that the $j$-th word will be repeated at the $n$-th step (the normalization is chosen for the purpose of convenience as it should become clear later on). 
When the $j$-th word is a repetition, we set $N_j(n)=A_j(n)=0$. 

When the $j$-th word is new, the evolution of  its occurrence counting and of its attraction bear obvious similarities, especially when $n\gg 1$: if $w_j$  is repeated at the $n$-th step, the occurrence counting increases exactly by $1$ and the attraction by slightly less than $1$, and otherwise  the occurrence counting is unchanged whereas the attraction slightly decreases. In particular there is the identity
$$ N_j(n)= \mathrm {Card} \{0\leq k \leq n-1: A_{j}(k) < A_j(k+1)\}.$$

Our purpose here is to check that up-to a simple time-rescaling, occurrence counting and attraction converge jointly in distribution as the total length of the string goes to $\infty$. To describe their limits, we introduce a time-inhomogeneous Markov process with c\`adl\`ag paths in $[0,\infty)$,  $A=(A(t))_{t\geq 0}$. The law $P_u$ of $A$ depends on a parameter $u\in(0,1)$ which should be thought of as a birth-time. Specifically, under $P_u$, we have that $A(t)=0$ for $t<u$ and $A(u)=1$,  the (right) slope of the path at time $t>0$ is  $-\alpha A(t)/t$, and further for $t\geq u$, jumps of size $1$ occur with intensity $\bar p(\alpha+1)A(t)/t$.  In other words, its infinitesimal generator at time $t\geq u$ is given by
\begin{eqnarray}\label{E:genA}
& &\lim_{\varepsilon\to 0+} \frac{E_u\left(f(A(t+\varepsilon))\mid A(t)=a\right)-f(a)}{\varepsilon}\\ \nonumber
&=&-\alpha \frac{a}{t} f'(a) + 
\bar p(\alpha+1)\frac{a}{t}  (f(a+1)-f(a)),
\end{eqnarray}
for any $a>0$ and any bounded and continuously differentiable function $f: (0,\infty)\to \R$. The slope and jump rates being bounded, the existence and uniqueness of this process are immediate.
We further denote the counting process of the jumps of $A$ by
$$N(t)=\mathrm {Card} \{0< s \leq t: A(t)=A(t-)+1\}.$$

We shall now check that the distribution of the pair of processes $(A,N)=(A(t), N(t))_{t\geq 0}$ under $P_u$ arises as  the weak limit of a time-rescaled version of the attraction and occurence counting.

\begin{lemma} \label{L1}
\begin{enumerate}
\item[(i)] Let $(j(n))_{n\geq 1}$  be a sequence with  $j(n)/n\to u \in(0,1)$. The conditional distribution of the pair of  time-rescaled processes 
$$(A_{j(n)}(\lfloor t n\rfloor),N_{j(n)}(\lfloor t n\rfloor))_{ t\geq 0}$$
 given that the $j(n)$-th word is new 
converges in the sense of Skorohod  towards the law of $(A,N)$ under $P_u$.

\item[(ii)]  Let $(k(n))_{n\geq 1}$ be a second sequence with $k(n)/n\to v\in(0,1)$ and $j(n)\neq k(n)$ for all $n$ sufficiently large.
The joint conditional distribution of the two pairs of time-rescaled counting processes 
$$(A_{j(n)}(\lfloor t n\rfloor),N_{j(n)}(\lfloor t n\rfloor))_{ t\geq 0} \quad \text{and} \quad  (A_{k(n)}(\lfloor t n\rfloor),N_{k(n)}(\lfloor t n\rfloor))_{ t\geq 0}$$
 given  that the $j(n)$-th and the $k(n)$-th words are both new 
converges in the sense of Skorohod  towards the law of two pairs of independent processes 
distributed as  $(A,N)$ under $P_u$ and under $P_v$, respectively.

\end{enumerate}

\end{lemma} 
\begin{proof} 
 We shall only prove the first assertion, the argument for the second is similar but with heavier notation and details are left to scrupulous readers.

We see from from the dynamics of the $q$-weighted Simon's model that, given that the $j$-th word is new,  the pair $((A_j(\ell), N_j(\ell)): \ell \geq j)$ is an inhomogeneous Markov chain started at time $j$ from $(1,1)$, with probability transitions given for any  $\ell\geq j$, $a>0$ and $k\geq 1$ by 
\begin{eqnarray*}
& &\PP\left( A_j(\ell+1) =\frac{\ell^{\alpha} a}{(\ell+1)^{\alpha}} +1, N_j(\ell+1) = k+1 \mid A_j(\ell)= a,  N_j(\ell)=k\right)\\
&=& \bar p  \ell^{\alpha}a/s_{\ell}
\end{eqnarray*}
and 
\begin{eqnarray*}
& &\PP\left( A_j(\ell+1) =\frac{\ell^{\alpha} a}{(\ell+1)^{\alpha}}, N_j(\ell+1) = k \mid A_j(\ell)= a,  N_j(\ell)=k\right)\\
&=& 1-\bar p  \ell^{\alpha}a/s_{\ell},
\end{eqnarray*}
where 
$$s_{\ell}=\sum_{i=1}^{\ell} i^{\alpha} \sim \ell^{\alpha+1}/(\alpha +1).$$

Then take $j=j(n)$ and $\ell=\ell(n)\sim tn$. 
It follows that for every bounded function $f:(0,\infty)\times \N \to \R$ which is continuously differentiable in the first variable,
\begin{eqnarray*}
&\lim_{n\to \infty}& n \E\left(f((A_{j}, N_{j})(\ell+1))- f(a,k)\mid (A_{j}, N_{j})(\ell)=(a,k)\right)\\
&=&-\alpha \frac{a}{t} \frac{\partial f}{\partial a}(a,k) + 
\bar p(\alpha+1)\frac{a}{t}  (f(a+1,k+1)-f(a,k)).
\end{eqnarray*}
On the other hand, the process $(A,N)$ under $\PP_u$ is an inhomogeneous Feller process, 
with $(A(t), N(t))=(0,0)$ for all $t<u$, $(A(u), N(u))=(1,1)$, 
and thanks to \eqref{E:genA},  its infinitesimal generator given by 
the right-hand side above. Our claim can now be derived from basic Markov chains approximation; see e.g. Theorem 19.28 in \cite{Kall}.
\end{proof}

\section{Connection to continuous state and Crump-Mode-Jagers branching processes}
The time-inhomogeneity of the Markov process $A$ introduced in Section 2 is essentially artificial, in the sense that $A$ actually results from a time-homogeneous Markov process by a deterministic logarithmic time substitution, as we shall now explain. Consider a time homogeneous Markov process $Z=(Z(t))_{t\geq 0}$ on $(0,\infty)$  with infinitesimal generator given for any smooth function $f:(0,\infty)\to \R$ by
\begin{equation} \label{E:genZ}
{\mathcal G}f(x)= -bx f'(x) + x(f(x+1)-f(x)),
\end{equation}
where $b\geq 0$ is some parameter.  We shall always deal with the situation when the process starts from $Z(0)=1$, and shall not  mention the initial state any further. 
The following observation is immediate by comparing \eqref{E:genZ} with \eqref{E:genA}. 

 \begin{lemma} \label{L2} 
 Take \begin{equation}\label{E:b}
 b=\frac{\alpha}{\bar p (\alpha+1)}.
 \end{equation}
For any fixed $u\in(0,1)$, the process given  by 
\begin{equation}\label{E:AZ}
\left\{ \begin{matrix} 0 \qquad &\text{for } t< u,\\
Z(\alpha b^{-1}\ln (t/u))\qquad &\text{for } t\geq u,\\
\end{matrix} \right.
\end{equation}
 has the law $P_u$. 
\end{lemma}

The process $Z$ is a simple instance of a continuous state branching process (in short, CSBP). We now recall some basic features in this setting and refer to Section VI.6 in \cite{AtNey} or Chapter 12 in \cite{Kypr} for background. Lamperti described a construction of general CSBPs from L\'evy processes  with no negative jumps (see Theorem 12.2 in \cite{Kypr}),  which we specialize to our setting.

Consider first  a standard Poisson process $(\eta_t)_{t\geq 0}$  started from $\eta_0=1$, and then
$$\xi_t=\eta_t-bt \qquad \text{ for every }t\geq 0.$$
 So $(\xi_t)_{t\geq 0}$ is a L\'evy process with no negative jumps   started from $\xi_0=1$,
and more precisely, all its jumps have unit size. Next   write 
$$\zeta=\inf\{t>0: \xi_t=0\}$$ for the first hitting time $0$, with the usual convention $\inf\varnothing = \infty$. 
 The indefinite integral $t\mapsto \int_0^t \dd s/\xi_s$ yields a bijection from $[0,\zeta)$ to $[0,\infty)$; we denote the inverse bijection  by $T(\cdot)$. Then the time-changed process $(\xi_{T(t)})_{t\geq 0}$ is  a version of $Z$. 
 
 Conversely,  one notes the identity
 $$T(t)=\int_0^t Z(s)\dd s,$$
 so that if we define
\begin{equation}\label{E:eta}
B(t)= Z(t)+b \int_0^t Z(s)\dd s,
\end{equation}
then  we can identify $B(t)=\eta_{T(t)}$. That is, $B(t)$ is  the number of birth events (i.e. the number of jumps of $Z$) occurring on the time-interval $[0,t]$, agreeing that $t=0$ is viewed as the first birth event (i.e. jump time of $Z$). We now derive from Lemma \ref{L2} the following:
 
\begin{cor} \label{C2} Let $b$ be given by \eqref{E:b} and fix any $u\in(0,1)$. In the notation above, the process defined by 
\begin{equation}\label{E:AZ}
\left\{ \begin{matrix} 0 \qquad &\text{for } t< u,\\
B(\alpha b^{-1}\ln (t/u))) \qquad &\text{for } t\geq u,\\
\end{matrix} \right.
\end{equation}
 has the same  law as $N=(N(t))_{t\geq 0}$ under $P_u$. 
\end{cor}
   
   As we shall see later on, the tail estimates in Theorem \ref{T1}(ii-iii) depend crucially on 
  the behavior of $N(1)$ under $P_u$ as $u\to 0+$.   Corollary \ref{C2} enables us to translate this question in terms of the asymptotic behavior of $B(t)$ as $t\to \infty$. In this direction, it is well-known 
 that, loosely speaking, the large time asymptotic behavior of a supercritical branching process depends on integrability properties of a remarkable martingale. In the present setting, the CSBP $Z$ is supercritical if and only if $b<1$, and we claim the following. 
 
\begin{lemma} \label{L3}
 When $b<1$, the process 
 $$W_t=\e^{-(1-b)t} Z(t), \qquad t\geq 0$$
  is a martingale which is bounded in $L^k(\PP)$ for all $k\geq 1$. 
\end{lemma}

 \begin{proof} That $W$ is  a martingale is well-known and indeed immediate from \eqref{E:genZ}. We shall now check by induction that for every $k\geq 1$, there exists some constant $c_k<\infty$ such that
 \begin{equation} \label{E:rec}
 \E(Z(t)^k)  \leq  c_k \e^{k(1-b)t}\qquad \text{for all }t\geq 0.
 \end{equation} 
 For $k=1$, \eqref{E:rec} is actually an equality with $c_1=1$, thanks to the martingale property of $W$. Let us now assume that  for some $\ell \geq 2$,  \eqref{E:rec}  holds for all $k=1, \ldots, \ell-1$. Take $f(x)=x^{\ell}$, so that  
 $${\mathcal G}f(x) = -\ell b x^{\ell}  +x \sum_{j=0}^{\ell-1} {\ell \choose j}  x^j
 = \ell(1-b) x^{\ell} + \sum_{j=0}^{\ell-2} {\ell \choose j}  x^{j+1}. $$
Combining Kolmogorov's forward equation
 $$ \frac{\dd }{ \dd t} \E(f(Z(t)))=  \E({\mathcal G}f(Z(t)))$$
 with  \eqref{E:rec}, 
 we deduce  that for some $\gamma>0$, there is the inequality
  $$ \frac{\dd }{ \dd t} \ln  \E(Z(t)^{\ell}) \leq  \ell(1-b)   +  \gamma \frac{ \e^{(1-b)(\ell-1)t}}{\E(Z(t)^{\ell})} . $$
 
 On the other hand, we know from Jensen's inequality that 
 $$\E(Z(t)^{\ell})\geq \exp(\ell(1-b)t)\qquad \text{for all }t\geq 0,$$
 and hence 
 $$\int_{0}^{\infty} \frac{ \e^{(1-b)(\ell-1)t}}{\E(Z(t)^{\ell})} \dd t \leq \int_{0}^{\infty}  \e^{-(1-b)t} \dd t= 1/(1-b).$$
 We conclude that 
 $$\E(Z(t)^{\ell})\leq 
 \e^{\gamma/(1-b)} \exp(\ell(1-b)t)\qquad \text{for all }t\geq 0.$$
So \eqref{E:rec} also holds for $k=\ell$ and the proof is complete. 
   \end{proof} 
      
 We write $W_{\infty}$ for  the terminal value of the martingale $W$ in Lemma \ref{L3} and $W^*=\sup_{t\geq 0} W_t$ for its overall supremum. We immediately deduce the following strong limit theorem 
  for  the number of birth events $B$. 
  
   \begin{cor} \label{C3}  When $b<1$, there is the convergence 
   $$\lim_{t\to \infty} \e^{-(1-b)t}B(t) = \frac{1}{1-b} W_{\infty}, \qquad \text{a.s.}$$
   Further,
   $$\sup_{t\geq 0} \e^{-(1-b)t}B(t) \leq \frac{1}{1-b} W^* \in L^k(\PP) \qquad \text{for all }k\geq 1.$$

   \end{cor}
  \begin{proof}  Writing 
  $$\e^{-(1-b)t}\left( Z(t)+ b \int_0^t Z(s)\dd s\right)= W_t + b\int_0^{\infty} \Ind_{\{s\leq t\}}  W_{t-s} \e^{-(1-b)s} \dd s,$$
yields the first claim by dominated convergence. The inequality of the second claim is obvious 
and the assertion  that $W^* \in L^k(\PP)$ follows from Lemma \ref{L3} and Doob's inequality. \end{proof}

 \begin{rem} In the case $b=0$, $Z$ is the standard Yule process, and for each $t$, $Z(t)$ has thus the geometric distribution with parameter $\e^{-t}$. Theoretically, the calculation in the proof of Lemma \ref{L3} allows us to compute inductively the entire moments of $Z(t)$ for any $b> 0$, and then also those of $B(t)$ by standard methods. Because $Z$ is stochastically dominated by a Yule process, the moment problem is determinate. So in theory, this approach enables to characterize the law of $B(t)$; however I have not been able to get an explicit formula.
 \end{rem}

 The final lemma of this section deals with the sub-critical case, and will provide the key to Theorem \ref{T1}(iii).

\begin{lemma} \label{L4}  When $b> 1$,  $Z$ is subcritical, $\lim_{t\to \infty}Z_t=0$ a.s. and 
one has
 $$\E \left( \exp\left( cb \int_0^{\infty} Z(s)\dd s\right) \right) = b $$ 
 for $c=\ln  b+b^{-1} -1>0$. 
 \end{lemma}

 \begin{proof} The first two assertions should be plain, and we  just need to establish the displayed identity.
 To start with, we recall from Lamperti's transformation that the integral $\int_0^{\infty} Z(s)\dd s$
 coincides with  the hitting time $\zeta$  of $0$ by the L\'evy process $\eta$.

 The Laplace exponent $\psi$ of the L\'evy process $\xi$ is $\psi(\lambda)=b \lambda-1+\e^{-\lambda}$ for $\lambda\in\R$, meaning that $\E(\exp(-\lambda (\xi_t-1)))=\exp( t\psi(\lambda))$. This function reaches its minimum at 
 $-\ln  b<0$, with $\min \psi = \psi(-\ln  b)=-b \ln  b -1+b<0$. 
 It follows from a well-known formula for the Laplace transform of first-passage times (see e.g. Theorem 3.12 in \cite{Kypr}) that
  $$\E \left( \exp((b\ln  b+1-b)\zeta \right) = \exp(\ln  b)=b,$$
    which establishes our claim.
         \end{proof} 
         
 We now conclude this section by pointing at a connection with another branching process, namely of Crump-Mode-Jagers (in short, C-M-J) type. Recall that a C-M-J branching process is a model for the evolution of a population in continuous time, where individuals beget children according to independent copies of a point process $\Xi$ on $(0,\infty)$ and locations of atoms are interpreted as birth of children. In other words, a C-M-J process can be constructed from a branching random walk on $[0,\infty)$ with reproduction law given by the distribution of $\Xi$, by viewing any atom, say located at $s$, at any generation of the branching random walk, as an individual born at time $s$ in the C-M-J process. Then individuals alive at time $t$ in the C-M-J process correspond to locations $s\leq t$ of atoms  in the branching random walk. 

Now take for $\Xi$ a Poisson point measure with intensity $\e^{-bt} \dd t$, and assign to each individual
in the C-M-J process a size which decays exponentially with time. Specifically the size  of an individual at age $a$ is $\e^{-ab}$, so that each individual in the C-M-J process begets children at rate given precisely by its current size. It follows that, if we further assume that the C-M-J  process starts from a single ancestor at time $0$, then
 the process describing the sum of sizes of individuals as a function of time, is a version of the CSBP $Z$.
 As a consequence, the process $(B(t))_{t\geq 0}$ has the same law as the process of the number of individuals alive of the population in this C-M-J process. In this framework, the so-called Malthus exponent  is readily identified with $1-b$, and the martingale $W$ is known as the intrinsic martingale. We refer to \cite{Nerman, Jagers} which provide in particular a criterion for  uniformly integrability of intrinsic martingales and demonstrate the importance of their roles in limit processes for C-M-J branching processes. Notably, the first part of Corollary \ref{C3} can also be derived from Theorem 5.4 in \cite{Nerman}.

\section{Proof of Theorem \ref{T1}}
We have now all the ingredients needed for the proof of Theorem \ref{T1}. To start with, recall that  for every $\ell \geq 1$,
$$\nu_n(\ell)= \sum_{i=1}^n \Ind_{\{N_i(n)=\ell\}}$$  denotes the number of different words which have been repeated exactly $\ell$ times when the string reaches the length $n$. 
Recall also from Section 2 the definition of the law $P_u$ for any $u\in(0,1)$, and of the counting process of jumps $N$.
We now define a probability mass function $\varphi$ on $\N$ by
$$\varphi(\ell)=\int_0^1 P_u(N(1)=\ell)\dd u \qquad \text{for every }\ell \geq 1.$$

With these notation at hand, we easily deduce Theorem \ref{T1}(i) from Lemma \ref{L1}
by first and second moment calculations,.

\begin{proof}[Proof of Theorem \ref{T1}(i)]  Let $U$ be a uniform random variable on $(0,1)$ independent of the weighted Simon's model.
For every $n\geq 1$, write $j(n)=\lceil nU\rceil$, so that $j(n)$ is uniformly distributed on $\{1, \ldots, n\}$ and $j(n)/n\to U$. 
Recall that for every $j\geq 2$, the probability that the $j$-th word is new equals $p$ and that $N_j(n)=0$ if the $j$-th word is a repetition. We immediately deduce from
Lemma \ref{L1}(i) that 
$$\lim_{n\to \infty} \frac{1}{np} \E\left(\sum_{j=1}^n \Ind_{\{N_j(n)=\ell\}}\right) =\int_0^1 P_u(N(1)=\ell)\dd u.$$

Next,  let $V$ be a second uniform random variable on $(0,1)$, independent of $U$ and the weighted Simon's model, and set $k(n)=\lceil nV\rceil$. So $k(n)$ is uniformly distributed on $\{1, \ldots, n\}$ and independent of $j(n)$, and writing
$$\left(\sum_{j=1}^n \Ind_{\{N_j=\ell\}}\right)^2 =\sum_{j=1}^n \Ind_{\{N_j(n)=\ell\}}\Ind_{\{N_k(n)=\ell\}}, $$
we deduce similarly from
Lemma \ref{L1}(ii) that 
$$\lim_{n\to \infty} \frac{1}{(np)^2} \E\left(\left(\sum_{j=1}^n \Ind_{\{N_j(n)=\ell\}}\right)^2 \right) =\left(\int_0^1 P_u(N(1)=\ell)\dd u\right)^2,$$
which establishes our claim. 
\end{proof}

The proof Theorem \ref{T1}(ii-iii) relies on the CSBP $Z$ introduced in Section 3; we henceforth let the parameter $b>0$ there be given by \eqref{E:b}. 

\begin{proof}[Proof of Theorem \ref{T1}(ii)] 
Thanks to Corollary \ref{C2}, we have 
 \begin{equation}\label{E:phiB}
 \sum_{j>k}  \varphi(j)= \int_0^1 P_u(N(1)>k)\dd u= \PP\left( B(\alpha b^{-1}\epsilon)  >k\right),
 \end{equation}
where $\epsilon$ is a random variable independent of  $B$ and with the standard exponential distribution. 

Next define for every $t\geq 0$ the variables
$$\beta^*(t)=\sup_{s\geq t} \e^{-(1-b)s}B(s)\quad \text{and} \quad \beta_*(t)=\inf_{s\geq t} \e^{-(1-b)s}B(s),$$
and observe from Corollary \ref{C3} that for any $k\geq 1$, we have by dominated convergence
\begin{equation} \label{E:bbb}
\lim_{t\to \infty} \beta_*(t) = \lim_{t\to \infty} \beta^*(t)= \frac{1}{1-b} W_{\infty} \qquad \text{in }L^k(\PP).
\end{equation}

We first treat the lower-bound. 
Fix $t>0$. By the definition of $\beta_*(t)$, the event $B(\alpha b^{-1}\epsilon)  >k$ holds whenever
$$\alpha b^{-1}\epsilon\geq t \quad \text{and} 
\quad \alpha b^{-1}\epsilon > (1-b)^{-1} \ln \left(k/\beta_*(t)\right),$$
and hence {\em a fortiori} whenever
$$\alpha b^{-1}\epsilon> (1-b)^{-1} \ln \left(k/\beta_*(t)\right) \geq t.$$
This yields the lower bound 
$$ 
 \PP\left( B(\alpha b^{-1}\epsilon)  >k\right) \geq \PP\left(\epsilon > \frac{b}{\alpha(1-b)} \ln \left(\frac{k}{\beta_*(t)}\right)\right) - \PP\left( \beta_*(t)> k \e^{(1-b)t}\right)
$$
On the one-hand, recalling that $b$ is given by \eqref{E:b} and that $\epsilon$ is exponentially distributed and independent of $B$, we get
$$
\PP\left(\epsilon > \frac{b}{\alpha(1-b)} \ln \left(\frac{k}{\beta_*(t)}\right)\right)= 
\E\left( \beta_*(t)^{1/(\bar p(1+\alpha)-\alpha)}\right)
 k^{-1/(\bar p(1+\alpha)-\alpha)} .
 $$
On the other hand, since $\beta_*(t)\leq (1-b)^{-1}W^*\in L^{\ell}(\PP)$ for any $\ell \geq 1$ (see Corollary \ref{C3}), 
we have by the Markov's inequality
$$\PP\left( \beta_*(t)> k \e^{(1-b)t}\right) = O(k^{-\ell}).$$
Putting the pieces together, we have shown that for any $t>0$,
$$\liminf_{k\to \infty}  k^{1/(\bar p(1+\alpha)-\alpha)}\PP\left( B(\alpha b^{-1}\epsilon)  >k\right)
\geq \E\left( \beta_*(t)^{1/(\bar p(1+\alpha)-\alpha)}\right).
$$
Letting $t\to \infty$ and using \eqref{E:bbb}, we conclude that
$$\liminf_{k\to \infty}  k^{1/(\bar p(1+\alpha)-\alpha)}\PP\left( B(\alpha b^{-1}\epsilon)  >k\right)
\geq \E\left( \left( \frac{W_{\infty}}{1-b}\right)^{1/(\bar p(1+\alpha)-\alpha)}\right).
$$

We now turn our attention to the upper-bound and 
fix $t>0$. By the definition of $\beta^*(t)$, there is the inclusion of events
$$\{B(\alpha b^{-1}\epsilon)  >k\} \subset \Lambda_1\cup \Lambda_2$$
with
$$\Lambda_1=\{ \alpha b^{-1}\epsilon \geq (1-b)^{-1} \ln \left(k/\beta^*(t)\right)\} \quad
 \text{and} 
\quad \Lambda_2=\{\alpha b^{-1}\epsilon \leq  t\}.$$
Using the identity $\Lambda_1\cup \Lambda_2=\Lambda_1\cup(\Lambda_1^{\rm c}\cap \Lambda_2)$ 
and then observing that
$$\Lambda_1^{\rm c}\cap \Lambda_2 \subset \{(1-b)^{-1} \ln \left(k/\beta^*(t)\right)<t\},$$
we see that $B(\alpha b^{-1}\epsilon)  >k$ implies 
$$ \text{either} 
\quad \alpha b^{-1}\epsilon > (1-b)^{-1} \ln \left(k/\beta^*(t)\right)\quad 
 \text{or } 
\quad (1-b)^{-1} \ln \left(k/\beta^*(t)\right)<t.$$
This yields the upper-bound 
$$ 
 \PP\left( B(\alpha b^{-1}\epsilon)  >k\right) \leq \PP\left(\epsilon \geq  \frac{b}{\alpha(1-b)} \ln \left(\frac{k}{\beta^*(t)}\right)\right) + \PP\left( \beta^*(t)> k \e^{(1-b)t}\right)
$$
By the same argument as for the lower-bound, we arrive at
$$\limsup_{k\to \infty}  k^{1/(\bar p(1+\alpha)-\alpha)}\PP\left( B(\alpha b^{-1}\epsilon)  >k\right)
\leq \E\left( \beta^*(t)^{1/(\bar p(1+\alpha)-\alpha)}\right).
$$
Letting $t\to \infty$ and using \eqref{E:bbb}, we conclude that
$$\limsup_{k\to \infty}  k^{1/(\bar p(1+\alpha)-\alpha)}\PP\left( B(\alpha b^{-1}\epsilon)  >k\right)
\leq \E\left( \left( \frac{W_{\infty}}{1-b}\right)^{1/(\bar p(1+\alpha)-\alpha)}\right).
$$
\end{proof}

\begin{rem} The proof above of Theorem \ref{T1}(ii) identifies the constant $C$ there as 
$$C=\E\left( \left( \frac{W_{\infty}}{1-b}\right)^{1/(\bar p(1+\alpha)-\alpha)}\right).$$
In this direction, recall that for $\alpha=0$, one has $b=0$, so $Z$ is simply the Yule process and $W_{\infty}$ follows the standard exponential distribution. This yields $C=\Gamma(1+1/\bar p)$, which agrees with what was known for Simon's original model. 
\end{rem}

Finally, the last part of Theorem \ref{T1} follows readily from Lemma \ref{L4}. 
\begin{proof}[Proof of Theorem \ref{T1}(iii)] We take $c=\ln  b+b^{-1} -1>0$ and write $N(\infty)$ for the limit as $t\to \infty$ of the counting process
$N$. We have from the definition of $\varphi$ and the inequality $N(1) \leq N(\infty) $ that 
 $$\sum_{j=1}^{\infty} \e^{cj}  \varphi(j)\leq \int_0^1 E_u(\e^{cN(\infty)})\dd u.$$
 Since $b>1$, we know from Corollary \ref{C2} and Lemma \ref{L4} that the law of $N(\infty)$ under $P_u$ is the same as that of $b\int_0^{\infty} Z_s \dd s$. Lemma \ref{L4} further shows that
$ E_u(\e^{cN(\infty)})=b$ for any $u\in(0,1)$, and we conclude that 
 $$\sum_{j=1}^{\infty} \e^{cj}  \varphi(j)\leq  b.$$
\end{proof}

To conclude this work, we observe that the characterization \eqref{E:phiB} of the limiting mass probability function $\varphi$ in terms of a branching process extends the fact that the Yule-Simon distribution can be realized as the law of a standard Yule process evaluated at an independent exponentially distributed random time with parameter $\rho=1/\bar p$.   It is also interesting to stress that similar variables -but associated to different branching processes- also arise in \cite{GvHW}, even though the framework of \cite{GvHW} and the present one seem to be rather different
 (in particular, the degree distribution plays a key role in there, whereas  we here are rather concerned with the distribution of the total population size).

\bibliographystyle{plain}
\bibliography{SimonMem}
\end{document}